\documentclass[leqno,11pt]{amsart}

\usepackage{
  amsmath,
  amssymb,
  mathtools,
  enumitem,
  tikz,
  stmaryrd
}

\usepackage{array}
\usepackage{placeins}

\usepackage[
  bookmarksnumbered,
  hidelinks,
  pdftitle={Combinatorial and algebraic proofs of Keller's A2 square A2 q-dilogarithm identity},
  pdfauthor={Richard Rimanyi}
]{hyperref}

\usepackage{geometry}
\newcommand{\E}{\mathbb{E}}
\newcommand{\C}{\mathbb{C}}
\newcommand{\N}{\mathbb{N}}
\newcommand{\Z}{\mathbb{Z}}
\newcommand{\cB}{\mathcal{B}}
\newcommand{\cP}{\mathcal{P}}
\newcommand{\Part}{\mathsf{P}}
\newcommand{\Dist}{\mathsf{D}}
\newcommand{\Rect}{\mathsf{R}}

\newcommand{\wt}{\operatorname{wt}}

\newcommand{\kp}{\vdash}
\newcommand{\blank}{\varnothing}

\DeclareMathOperator{\Hilb}{Hilb}
\DeclareMathOperator{\DBZ}{DBZ}

\theoremstyle{plain}
\newtheorem{thm}{Theorem}[section]
\newtheorem{prop}[thm]{Proposition}
\newtheorem{lem}[thm]{Lemma}
\newtheorem{cor}[thm]{Corollary}

\theoremstyle{definition}

\theoremstyle{remark}
\newtheorem{remark}[thm]{Remark}
\newtheorem{example}[thm]{Example}

\title[Proofs of Keller's $q$-dilogarithm  identity]
{Combinatorial and algebraic proofs of Keller's  
\texorpdfstring{$A_2\square A_2$}{A2 square A2} $q$-dilogarithm identity}

\author{R. Rim\'anyi}
\address{Department of Mathematics, University of North Carolina at Chapel Hill,
Chapel Hill, NC}
\email{rimanyi@email.unc.edu}


\begin{document}

\begin{abstract}
The classical Durfee-square argument gives a combinatorial proof of the
pentagon identity for the quantum dilogarithm. Just as the pentagon
identity is associated with the $A_2$ quiver, Keller's identity is
associated with the square-product quiver $A_2\square A_2$. Previous
proofs of Keller's identity use cluster categories or spectral sequences in rapid-decay equivariant cohomology. We give three proofs of Keller's identity: a generating-function proof, an explicit weight-preserving bijection on
colored partitions, and a standard-monomial proof using a four-colored
arc algebra. Their common mechanism is an iterated Durfee decomposition:
two possible pairings give horizontal and vertical decompositions, while
a third binary step accounts for the coupling factor. These constructions
provide a ``superpotential analogue'' of the Durfee-square argument.
\end{abstract}

\maketitle


\section{Introduction}
\label{s:intro}

The quantum dilogarithm formal power series
\[
  \E(z)=
  \sum_{n\geq 0} \frac{(-z)^n q^{n^2/2}}{(1-q)(1-q^2)\cdots(1-q^n)}
\]
satisfies remarkable identities. The ones discussed in the present paper come on two levels: some are associated to Dynkin quivers, and some are associated to {\em pairs} of Dynkin quivers. 
These identities occur in various parts of mathematics; a non-exhaustive list is given in Figure~\ref{fig:fig}. This table already displays the goal of our paper (see the question mark): finding the combinatorial interpretation of Keller's $A_2\square A_2$ identity for quantum dilogarithms. 

The combinatorial interpretation of the $A_2$ identity (a.k.a. the pentagon identity) is the beautiful {\em Durfee squares} argument of classical combinatorics \cite{SylvesterFranklin1882}. Hence, ever since \cite{Rimanyi2013, AllmanRimanyi2018} the author has been interested in finding the {\em superpotential analogue}, or {\em rapid-decay analogue} of this classical combinatorial argument, that is, a combinatorial argument that proves Keller's $A_2\square A_2$ identity. The paper \cite{RimanyiWeigandtYong2018} grew out of this interest, but instead went in different directions.

In this paper we give three complementary combinatorial and algebraic proofs of Keller's $A_2\square A_2$ identity. In Sections~\ref{s:Keller-statement}--\ref{s:alternating-path} we give a generating-function proof and construct an explicit weight-preserving bijection between the two families of colored-partition tuples indexing the coefficient expansions. In Section~\ref{s:standard-monomials} we introduce a four-colored arc algebra; two iterated standard-monomial constructions give bases of the same algebra and yield the third proof. Our proofs follow the blueprint of the topological proofs; in fact, the arc algebra explicitly displays the superpotential. 

\begin{figure}
\begin{center}
{\renewcommand{\arraystretch}{1.7}
    \begin{tabular}{ >{\centering\arraybackslash}m{5cm} >{\centering\arraybackslash}m{1cm} >{\centering\arraybackslash}m{7cm} } 
    Dynkin quiver, e.g., $A_2$ & $\longleftrightarrow$ & pair of Dynkin quivers, e.g., $A_2\square A_2$
   \\
   pentagon identity for $\E$ & $\longleftrightarrow$ & Keller's identity for $\E$
   \\
   cluster algebra & $\longleftrightarrow$ & cluster category
   \\
   equivariant cohomology spectral sequence &  $\longleftrightarrow$ &  rapid decay equivariant cohomology spectral sequence 
   \\
   cohomological Hall algebra & $\longleftrightarrow$ & cohomological Hall algebra with superpotential
   \\
   counting partitions using Durfee squares & $\longleftrightarrow$ & ?
    \end{tabular}}
\end{center}
\caption{Two levels of quantum dilogarithm identities}\label{fig:fig}
\end{figure}

\FloatBarrier

\subsection{The identities; the coupling phenomenon} Cf. \cite[\S1]{AllmanRimanyi2018}.
The $\E$-series satisfies the celebrated pentagon identity, or $A_2$ identity,
\[
\E(y_1)\E(y_2)
  =
\E(y_2)\E(y_{12})\E(y_1)
\]
for certain $q$-commuting variables $y_1, y_2, y_{12}$. After unfolding the $q$-commutativity this boils down to the explicit
\[
\cP_{\gamma_1}\cP_{\gamma_2}=
\sum_{(r,s,k) \kp (\gamma_1,\gamma_2)}
q^{rs}
\cP_{r}\cP_{s}\cP_{k}
\qquad
(\forall \gamma_1,\gamma_2\in \N),
\]
where $\cP_n=1/\prod_{i=1}^n(1-q^i)$ and see \eqref{eq:kp} for the meaning of $\kp$.
Formal manipulations on this identity yield 
\begin{equation}
\label{eqn:55term}
\sum_{\substack{(r,s,k)\kp(\gamma_1,\gamma_2)\\
                (u,v,l)\kp(\gamma_3,\gamma_4)}}
q^{rs+uv}\cP_r\cP_s\cP_k\cP_u\cP_v\cP_l
=
\sum_{\substack{(r,s,k)\kp(\gamma_1,\gamma_3)\\
                (u,v,l)\kp(\gamma_2,\gamma_4)}}
q^{rs+uv}
\cP_r\cP_s\cP_k\cP_u\cP_v\cP_l.
\end{equation}
for all $\gamma_1,\gamma_2,\gamma_3,\gamma_4$. In fact both sides are equal to $\cP_{\gamma_1}\cP_{\gamma_2}\cP_{\gamma_3}\cP_{\gamma_4}$. The novelty of Keller's $A_2 \square A_2$ identity is that in \eqref{eqn:55term} one can insert an {\em extra coupling factor} $q^{kl}$ in each term, namely
\begin{equation}
\label{eqn:55termKeller}
\sum_{\substack{(r,s,k)\kp(\gamma_1,\gamma_2)\\
                (u,v,l)\kp(\gamma_3,\gamma_4)}}
q^{rs+uv+kl}\cP_r\cP_s\cP_k\cP_u\cP_v\cP_l
=
\sum_{\substack{(r,s,k)\kp(\gamma_1,\gamma_3)\\
                (u,v,l)\kp(\gamma_2,\gamma_4)}}
q^{rs+uv+kl}\cP_r\cP_s\cP_k\cP_u\cP_v\cP_l.
\end{equation}
In fact this coupled identity \eqref{eqn:55termKeller}, not the uncoupled \eqref{eqn:55term}, is the one that translates to an identity of quantum dilogarithm series:
\[
\E(y_2)\E(y_3)\E(y_{12})\E(y_{34})\E(y_1)\E(y_4)
	= \E(y_1)\E(y_4)\E(y_{13})\E(y_{24})\E(y_2)\E(y_3),
\]
for certain $q$-commuting $y$-variables. The combinatorics of the $q$-commutativity is dictated by the $A_2\square A_2$ quiver, see \eqref{eq:A2A2}.

The common mechanism behind the three proofs is a four-color, iterated
version of the Durfee decomposition. At each of the three
levels---generating series, signed sets, and standard monomials---the
horizontal pairings $(1,2),(3,4)$ and the vertical pairings
$(1,3),(2,4)$ give two decompositions of the same object. In either
decomposition, a third binary step couples the two intermediate colors;
its Durfee shift is precisely the factor $q^{kl}$ in
\eqref{eqn:55termKeller}. In the standard-monomial realization, the new
content is this four-color refinement, the two iterated normal forms,
and the identification of the third-node shift with Keller's coupling
factor.

\subsection{Structure of the paper}
The paper is organized as follows.

After Section~\ref{s:notation} fixes conventions and notation,
Sections~\ref{s:pentagon}--\ref{s:pentagon-algebra} recall facts and
proofs concerning the pentagon identity.

Section~\ref{s:uncoupled} discusses the uncoupled four-vertex identity,
an innocent generalization of the pentagon identity, as preparation for
the next main part.

Sections~\ref{s:Keller-statement}--\ref{s:alternating-path} give two
proofs of Keller's $A_2 \square A_2$ identity for $\E$, one using
generating functions and one using partition combinatorics.

Finally, Section~\ref{s:standard-monomials} gives a third proof of
Keller's identity by counting elements in a standard-monomial basis of
an arc algebra.

\subsection{Acknowledgements}
The author was supported by the U.S. National Science Foundation
under Grant No. 2152309. Any opinions, findings, and conclusions or recommendations expressed in this
material are those of the author(s) and do not necessarily reflect the views of the NSF.

\section{Notation and conventions}
\label{s:notation}

\noindent
We use the convention $\N=\{0,1,2,\ldots\}$.

\subsection{Quantum numbers, quantum dilogarithm}
In this paper we will consider formal power series in the variable $q^{1/2}$. Define 
\[
  (z;q)_n=\prod_{i=0}^{n-1}(1-zq^i),\qquad
  (z;q)_\infty=\prod_{i\geq0}(1-zq^i),
\]
and abbreviate $(q;q)_n$ to $(q)_n$. The $q$-binomial theorem in the
form used below is
\begin{equation}
\label{eq:q-binomial}
  \sum_{n\geq0}\frac{(z;q)_n}{(q)_n}X^n
  =
  \frac{(zX;q)_\infty}{(X;q)_\infty}.
\end{equation}
When we write a rational function in $q$, we mean its formal power-series
expansion at $q=0$. For example, define 
\[
  \cP_n=\frac{1}{(q)_n}
  =\frac{1}{(1-q)(1-q^2)\cdots(1-q^n)},\qquad \cP_0=1.
\]
We will also use the quantum dilogarithm series (with the convention from \cite{AllmanRimanyi2018}):
\[
  \E(z)=\sum_{r\geq 0} (-z)^rq^{r^2/2}\cP_r=
  \sum_{r\geq 0} \frac{(-z)^r q^{r^2/2}}{(1-q)(1-q^2)\cdots(1-q^r)}
  =
  (q^{1/2}z;q)_{\infty},
\]
which is a formal power series in the variable $z$, with coefficients in $\Z\llbracket q^{1/2} \rrbracket$.

We write
\begin{equation}\label{eq:kp}
  (r,s,k)\kp(a,b)
  \quad\Longleftrightarrow\quad
  r+k=a,\quad s+k=b,
\end{equation}
where $a,b,r,s,k$ are non-negative integers.

\subsection{Partition generating sequences, colored partitions}
Let 
\[
\Part_r=\{\text{weakly decreasing $r$-tuples of non-negative integers}\},
\]
 \[
\Dist_r=\{\text{strictly decreasing $r$-tuples of non-negative integers}\}.
\]
Defining $|\ |$ to be the sum of components, we have
\[
\sum_{\lambda\in\Part_r} q^{|\lambda|}=\cP_r, \qquad 
\sum_{\delta\in\Dist_r} q^{|\delta|} = \frac{q^{\binom{r}{2}}}{(q)_r}.
\]
We will decorate partitions with {\em colors}, as follows. 
Colors will be monomials in commuting variables $x_1,x_2,x_3,x_4$.   If $\lambda\in\Part_r$ has color $A$, define its weight 
\[
  \wt(\lambda^A)=A^r q^{|\lambda|}.
\]
In notation we occasionally use obvious abbreviation of colors, say, $12$ abbreviates the color $x_1x_2$ and $1234$ abbreviates
$x_1x_2x_3x_4$.

Write $\Part(A)=\bigsqcup_{r\geq0}\Part_r(A)$ for all ordinary
$A$-colored partitions. Write $\Dist(A)$ for the same construction with distinct parts. For $\lambda \in \Part_r(A)$ (resp. $\delta \in \Dist_r(A)$) we write $\ell(\lambda)=r$ (resp. $\ell(\delta)=r$). We make $\Dist(A)$ into a `signed' set by defining
$\operatorname{sgn}(\delta^A)=(-1)^{\ell(\delta)}$.
Consequently, these classes model the Pochhammer factors underlying
$\E$ and $\E^{-1}$.  More precisely, after the substitution
$A=q^{1/2}z$, the signed set $\Dist(A)$ models $\E(z)$ and
$\Part(A)$ models $\E(z)^{-1}$.  Explicitly,
\begin{equation}
\label{eq:colored-gen}
  \sum_{\lambda\in\Part(A)}\wt(\lambda)
   =\frac{1}{(A;q)_\infty},
  \qquad
  \sum_{\delta\in\Dist(A)}
       \operatorname{sgn}(\delta)\wt(\delta)
   =(A;q)_\infty.
\end{equation}



\section{Pentagon identity in quantum dilogarithm and coefficient forms}
\label{s:pentagon}
Consider $q$-commuting variables $y_1,y_2$, that is, we assume that $y_2y_1=qy_1y_2$. Define 
\[
y_{12}=-q^{1/2}y_1y_2.
\]
The pentagon identity for quantum dilogarithms is the identity
\begin{equation}
\label{eq:E-pentagon}
  \E(y_1)\E(y_2)
  =
  \E(y_2)\E(y_{12})\E(y_1).
\end{equation}
After rewriting both sides in the ordered monomial basis, for the coefficients of 
$y_1^ay_2^b$ on the two sides we obtain the identity 
\begin{equation}
\label{eq:pentagon}
  \cP_a\cP_b
  =
  \sum_{(r,s,k)\kp(a,b)}
  q^{rs}\cP_r\cP_s\cP_k.
\end{equation}
That is, the $\E$-identity \eqref{eq:E-pentagon} packages the infinitely many identities \eqref{eq:pentagon} (one for each $a,b$-pair) in a compact form.  

Identity \eqref{eq:E-pentagon} is the $A_2$ case of the quantum
dilogarithm identities for Dynkin quivers
\cite{Reineke2010,Rimanyi2013}.  It also has interpretations in
Donaldson--Thomas theory and wall crossing; see, for example,
\cite{KontsevichSoibelman2014,Keller2011}, and many other parts of mathematics. Our concern here is its
partition-combinatorial interpretation.

\section{Pentagon identity and Durfee rectangle arguments}
\label{s:durfee}

\subsection{The classical Durfee rectangle identity by stabilization}

Fix an integer $c\geq0$.  Specializing \eqref{eq:pentagon} to
$a=N$, $b=N+c$, and setting $d=N-k$, we obtain
\begin{equation}
\label{eq:finite-rectangle}
  \cP_N\cP_{N+c}
  =
  \sum_{d=0}^N
  q^{d(d+c)}\cP_d\cP_{d+c}\cP_{N-d},
\end{equation}
equivalently, 
\[
  \cP_{N+c}
  =
  \sum_{d=0}^N
  q^{d(d+c)}\cP_d\cP_{d+c}
  \frac{\cP_{N-d}}{\cP_N}.
\]
For fixed $c,d$ and every fixed power of $q$, the quotient
$\cP_{N-d}/\cP_N$ stabilizes to $1$ as $N\to\infty$, while
$\cP_{N+c}$ stabilizes to $1/(q;q)_\infty$.  Since only finitely many
$d$ contribute to any fixed power of $q$, we obtain
\begin{equation}
\label{eq:classical-Durfee-rectangle}
  \frac{1}{(q;q)_\infty}
  =
  \sum_{d\geq0}q^{d(d+c)}\cP_d\cP_{d+c}.
\end{equation}
This is the Durfee rectangle identity.  The ordinary Durfee square
is the case $c=0$ of this construction.  For fixed $c\geq0$, the
$n$th coefficient of the left-hand side is the number of partitions of
$n$, and the right-hand side counts the same objects according to the
largest possible top-left-aligned rectangle having $d$ rows and $d+c$
columns.  What remains is a partition to its right with at most $d$
rows and a partition below it with at most $d+c$ columns, as in the
figure below.
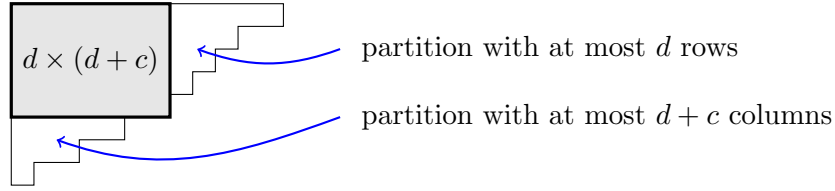
\[
\begin{tikzpicture}[scale=0.3]
  \draw (0,0) -- (12,0) -- (12,-1) -- (10,-1) -- (10,-2) -- (9,-2) -- (9,-3) --  (8,-3) -- (8,-4) -- (7,-4); 
  \draw (0,-5) -- (0,-8) -- (1,-8) -- (1,-7) -- (3,-7) -- (3,-6) -- (5,-6) -- (5,-5) -- (7,-5); 
  \fill[gray!20] (0,0) rectangle (7,-5);
  \draw[very thick] (0,0) rectangle (7,-5);
  \draw (3.5,-2.5) node {$d \times (d+c)$};
  \draw (15,-5) node[right] {partition with at most $d+c$ columns};
  \draw[thick, blue,->] (14.5,-5) to[out=-160,in=-20]  (2,-6);
  \draw (15,-2) node[right] {partition with at most $d$ rows};
  \draw[thick, blue,->] (14.5,-2) to[out=-160,in=-20]  (8.2,-2);
\end{tikzpicture}
\]

\subsection{A finite two-partition bijection}
We saw that after specialization, division by $\cP_N$, and stabilization, identity \eqref{eq:pentagon} can be interpreted as a bijection between sets of combinatorial objects. However, the same is true for \eqref{eq:pentagon} without specialization, division, and stabilization. 

There is an explicit weight-preserving bijection
\begin{equation}
\label{eq:finite-DBZ-map}
  \DBZ_{a,b}:
  \Part_a\times\Part_b
  \longrightarrow
  \bigsqcup_{(r,s,k)\kp(a,b)}
  \bigl(\Rect_{r,s}\times
        \Part_r\times\Part_s\times\Part_k\bigr),
\end{equation}
which we will call the \emph{finite Durfee--BZ bijection}, where $\Rect_{r,s}$ is the singleton rectangle of weight $q^{rs}$.
Thus the left-hand side of \eqref{eq:pentagon} counts two partitions, while a
right-hand object consists of a rectangle and three partitions. 
The construction is given in Section \ref{s:algorithm-Z}.

\begin{remark}
    Using our conventions for partitions, $\DBZ_{a,b}$ is closely related to a correspondence given in \cite[Prop.~3.1]{BressoudZeilberger1989}. That correspondence is commonly called Algorithm Z; see, for example, \cite{ChenChenFuZang2010}.  The finite Durfee--BZ bijection constructed below is obtained from this local correspondence by the involution principle; it is not itself stated in \cite{BressoudZeilberger1989}. It is an alternating-path bijection between two fixed-point sets; in this sense it is a small, `two-color prototype' of our result in Section~\ref{s:alternating-path}.
\end{remark}

\begin{remark}
When $a=N$ and $b=N+c$, the rectangle in
\eqref{eq:finite-DBZ-map} is $d\times(d+c)$, with $d=N-k$.
The third partition records the finite information that disappears from
\eqref{eq:finite-rectangle} after division and stabilization.  Thus
\eqref{eq:finite-DBZ-map}, rather than
\eqref{eq:classical-Durfee-rectangle}, is the form that can be iterated
without discarding data.
\end{remark}

\section{Algebraic proof of the pentagon identity}
\label{s:pentagon-algebra}

Introduce the two-variable series
\[
  F(X,Y)
  =
  \sum_{r,s\geq0}
  q^{rs}\cP_r\cP_sX^rY^s.
\]

\begin{lem}
\label{lem:F-product}
We have
\[
  F(X,Y)
  =
  \frac{(XY;q)_\infty}
       {(X;q)_\infty(Y;q)_\infty}.
\]
\end{lem}

\begin{proof}
Summing first over $s$ and using \eqref{eq:q-binomial} we have
\begin{equation*}
F(X,Y)
=
\sum_{r\geq0}\frac{X^r}{(q)_r}
  \frac{1}{(Yq^r;q)_\infty}
=
\frac{1}{(Y;q)_\infty}
\sum_{r\geq0}\frac{(Y;q)_r}{(q)_r}X^r
=
\frac{(XY;q)_\infty}
     {(X;q)_\infty(Y;q)_\infty}.
\end{equation*}
\end{proof}

From the lemma we obtain
\[
 \frac{1}{(X;q)_\infty(Y;q)_\infty}
 =
 F(X,Y)\frac{1}{(XY;q)_\infty}.
\]
Observe that the two sides are the generating series of the two sides of \eqref{eq:pentagon} over all $a,b$. Taking the coefficient of
$X^aY^b$ then proves \eqref{eq:pentagon}.



\section{The uncoupled identity}
\label{s:uncoupled}

Fix $\gamma=(\gamma_1,\gamma_2,\gamma_3,\gamma_4)\in\N^4$.  Applying
\eqref{eq:pentagon} independently to the pairs $(\gamma_1,\gamma_2)$ and
$(\gamma_3,\gamma_4)$ gives
\begin{equation}
\label{eq:uncoupled}
\sum_{\substack{(r,s,k)\kp(\gamma_1,\gamma_2)\\
                (u,v,l)\kp(\gamma_3,\gamma_4)}}
q^{rs+uv}\cP_r\cP_s\cP_k\cP_u\cP_v\cP_l
=
\sum_{\substack{(r,s,k)\kp(\gamma_1,\gamma_3)\\
                (u,v,l)\kp(\gamma_2,\gamma_4)}}
q^{rs+uv}\cP_r\cP_s\cP_k\cP_u\cP_v\cP_l.
\end{equation}
We call this the \emph{uncoupled four-vertex identity}.

\subsection{A combinatorial proof}

Let
\begin{equation}\label{eq:4factors}
  \mathcal U_\gamma
  =
  \Part_{\gamma_1}\times\Part_{\gamma_2}
  \times\Part_{\gamma_3}\times\Part_{\gamma_4}.
\end{equation}
Apply the finite Durfee--BZ bijection
\eqref{eq:finite-DBZ-map} to the first two and last two partitions.
This identifies \eqref{eq:4factors} with the objects counted on the left-hand side of \eqref{eq:uncoupled}.  Applying it instead to the first and third, and to the second and fourth, identifies the same set \eqref{eq:4factors} with
the objects counted on the right-hand side.  The composition
\[
  (\DBZ_{\gamma_1,\gamma_3}
   \times\DBZ_{\gamma_2,\gamma_4})
  \circ
  (\DBZ_{\gamma_1,\gamma_2}
   \times\DBZ_{\gamma_3,\gamma_4})^{-1}
\]
is therefore a weight-preserving bijection between the two sides.

\subsection{An algebraic proof}

By \eqref{eq:pentagon}, each side of \eqref{eq:uncoupled} equals
\begin{equation}
\label{eq:uncoupled-common}
  \cP_{\gamma_1}\cP_{\gamma_2}
  \cP_{\gamma_3}\cP_{\gamma_4}.
\end{equation}
Equivalently, summing over $\gamma$, the horizontal expression is
\[
\frac{F(x_1,x_2)}{(x_1x_2;q)_\infty}
\cdot
\frac{F(x_3,x_4)}{(x_3x_4;q)_\infty}
=
\prod_{i=1}^4\frac{1}{(x_i;q)_\infty},
\]
and the vertical factorization gives the same product.

\begin{remark}
The uncoupled identity is deliberately elementary.  The two applications
of the finite Durfee bijection do not interact, and the simple common
expression \eqref{eq:uncoupled-common} remains visible. This identity is presented only as a comparison with the coupled version below. 
\end{remark}



\section{Keller's \texorpdfstring{$A_2\square A_2$}{A2 square A2} identity}
\label{s:Keller-statement}
Consider four variables $y_1, y_2,y_3,y_4$ with commutation relations
\[
y_2y_1=qy_1y_2,
\quad
y_1y_3=qy_3y_1,
\quad
y_3y_4=qy_4y_3,
\quad
y_4y_2=qy_2y_4,
\qquad
y_1y_4=y_4y_1,
\quad
y_2y_3=y_3y_2.
\]
and define 
\[
y_{12}=-q^{1/2}y_1y_2,
\quad
y_{13}=-q^{1/2}y_3y_1,
\quad
y_{34}=-q^{1/2}y_4y_3,
\quad
y_{24}=-q^{1/2}y_2y_4.
\]
These commutation relations define a quantum torus and are read from the
underlying $A_2 \square A_2$ quiver
\begin{equation}\label{eq:A2A2}
\begin{tikzpicture}[->,semithick,auto,inner sep=1mm, baseline=10pt]
\node (1) at (0,1) {$1$};
\node (2) at (1,1) {$2$};
\node (4) at (1,0) {$4$};
\node (3) at (0,0) {$3$};
\path
(1) edge (3)
(3) edge (4)
(4) edge (2)
(2) edge (1);
\end{tikzpicture}
\end{equation}
see more details in \cite{KontsevichSoibelman2014, AllmanRimanyi2018}. In this fashion, the commutation relations of Section~\ref{s:pentagon} are associated to the quiver $1 \longleftarrow 2$.

\begin{thm}[Keller's $A_2 \square A_2$ identity, \cite{Keller2011,Keller2013,AllmanRimanyi2018}]
\label{thm:KellerE}
We have
\begin{equation}
\label{eq:E-Keller}
\E(y_2)\E(y_3)\E(y_{12})\E(y_{34})\E(y_1)\E(y_4)
=
\E(y_1)\E(y_4)\E(y_{13})\E(y_{24})\E(y_2)\E(y_3).
\end{equation}
\end{thm}

Using the notations 
\[
\begin{split}
(r,s,k;u,v,l)\kp_H\gamma
&\Longleftrightarrow
r+k=\gamma_1,\quad s+k=\gamma_2,\quad
u+l=\gamma_3,\quad v+l=\gamma_4,\\
(r,s,k;u,v,l)\kp_V\gamma
&\Longleftrightarrow
r+k=\gamma_1,\quad s+k=\gamma_3,\quad
u+l=\gamma_2,\quad v+l=\gamma_4,
\end{split}
\]
(here $H$ and $V$ refer to `horizontal' and `vertical').
We briefly explain the coefficient extraction in~\eqref{eq:E-Keller}.  Order monomials as
$y_1^{\gamma_1}y_2^{\gamma_2}y_3^{\gamma_3}y_4^{\gamma_4}$ and put
\[
  c(\gamma)
  =
  \frac12\sum_{i=1}^4\gamma_i^2
  -\gamma_1\gamma_3-\gamma_3\gamma_4.
\]
For example, using the quantum-torus relations,
\[
  \frac{(-y_{12})^kq^{k^2/2}}{(q)_k}
  =
  q^{k^2}\cP_k\,y_1^ky_2^k,
  \qquad
  \frac{(-y_{34})^lq^{l^2/2}}{(q)_l}
  =
  q^{l^2}\cP_l\,y_4^ly_3^l.
\]
After the remaining factors are put in the chosen order, the summand
indexed by $(r,s,k;u,v,l)\kp_H\gamma$ is
\[
  (-1)^{\gamma_1+\gamma_2+\gamma_3+\gamma_4}
  q^{c(\gamma)}
  q^{rs+uv+kl}
  \cP_r\cP_s\cP_k\cP_u\cP_v\cP_l\,
  y_1^{\gamma_1}y_2^{\gamma_2}
  y_3^{\gamma_3}y_4^{\gamma_4}.
\]
The same reordering on the right-hand side, now for
$(r,s,k;u,v,l)\kp_V\gamma$, gives the same prefactor
$(-1)^{\sum_i\gamma_i}q^{c(\gamma)}$ and again leaves
$q^{rs+uv+kl}$.  Cancelling this common prefactor gives the following
theorem.

\begin{thm}[Keller's coefficient identity]
\label{thm:Keller-coefficient}
For every $\gamma\in\N^4$,
\begin{equation}
\label{eq:Keller}
\sum_{(r,s,k;u,v,l)\kp_H\gamma}
q^{rs+uv+kl}
\cP_r\cP_s\cP_k\cP_u\cP_v\cP_l
=
\sum_{(r,s,k;u,v,l)\kp_V\gamma}
q^{rs+uv+kl}
\cP_r\cP_s\cP_k\cP_u\cP_v\cP_l.
\end{equation}
\end{thm}

Define $\mathcal H_\gamma$ to be the disjoint union, over
$(r,s,k;u,v,l)\kp_H\gamma$, of
\[
  \Rect_{r,s}\times\Rect_{u,v}\times\Rect_{k,l}
  \times
  \Part_r\times\Part_s\times\Part_k
  \times\Part_u\times\Part_v\times\Part_l.
\]
Define $\mathcal V_\gamma$ by the same formula with
$\kp_H$ replaced by $\kp_V$.  The weight is the product of the rectangle
and partition weights.  Thus the two sides of \eqref{eq:Keller} are the
generating functions of $\mathcal H_\gamma$ and $\mathcal V_\gamma$, see 
see
Figure~\ref{fig:iterated-durfee}.

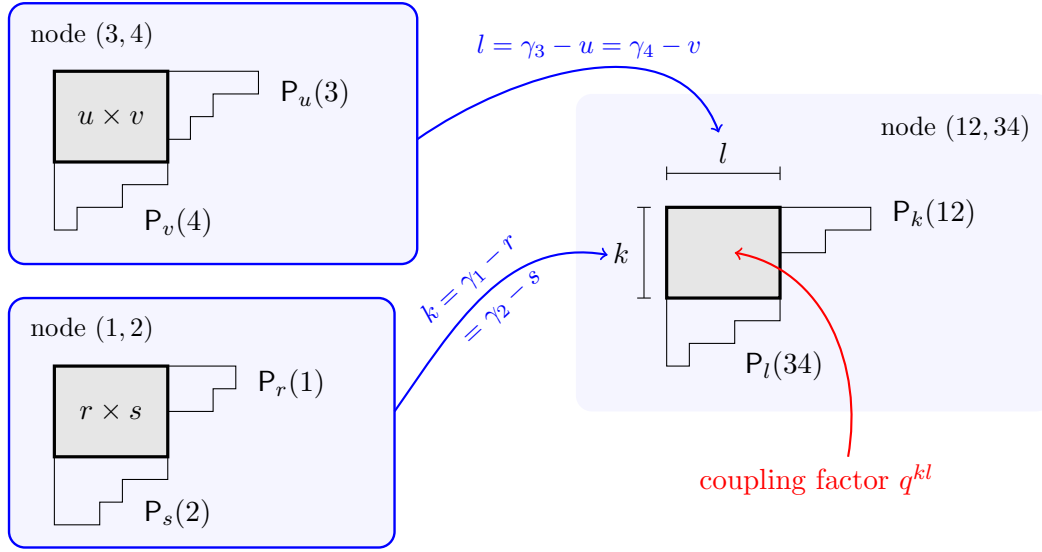
\begin{figure}
\begin{center}
\begin{tikzpicture}[scale=0.3]
  \draw[fill=blue!4, draw=blue!100, line width=.9pt, rounded corners=6pt]
  (-2,3) rectangle (16,-8.5);
  \draw[fill=blue!4, draw=blue!100, line width=.9pt, rounded corners=6pt] (-2,-10) rectangle (15,-21);
  \fill[blue!4, rounded corners=6pt] (23,-1) rectangle (44,-15);
  \draw (-1.5,1.6) node[right] {\small node $(3,4)$};
  \draw (-1.5,-11.4) node[right] {\small node $(1,2)$};
  \draw (43.5,-2.5) node[left] {\small node $(12,34)$};
  \fill[gray!20] (0,0) rectangle (5,-4);
  \draw (5,0) -- (9,0) -- (9,-1) -- (7,-1) -- (7,-2) -- (6,-2) -- (6,-3) -- (5,-3);
  \draw (0,-4) -- (0,-7) -- (1,-7) -- (1,-6) -- (3,-6) -- (3,-5) -- (5,-5) -- (5,-4);
  \draw[very thick] (0,0) rectangle (5,-4);
  \draw (2.5,-2) node {$u\times v$};
  \draw (9.5,-1) node[right] {$\Part_u(3)$};
  \draw (3.5,-6.7) node[right] {$\Part_v(4)$};
  \fill[gray!20] (0,-13) rectangle (5,-17);
  \draw (5,-13) -- (8,-13) -- (8,-14) -- (7,-14) -- (7,-15) -- (5,-15);
  \draw (0,-17) -- (0,-20) -- (2,-20) -- (2,-19) -- (3,-19) -- (3,-18) -- (5,-18) -- (5,-17);
  \draw[very thick] (0,-13) rectangle (5,-17);
  \draw (2.5,-15) node {$r\times s$};
  \draw (8.5,-13.8) node[right] {$\Part_r(1)$};
  \draw (3.5,-19.5) node[right] {$\Part_s(2)$};
  \fill[gray!20] (27,-6) rectangle (32,-10);
  \draw (32,-6) -- (36,-6) -- (36,-7) -- (34,-7) -- (34,-8) -- (32,-8);
  \draw (27,-10) -- (27,-13) -- (28,-13) -- (28,-12) -- (30,-12) -- (30,-11) -- (32,-11) -- (32,-10);
  \draw[very thick] (27,-6) rectangle (32,-10);
  \draw (36.5,-6.3) node[right] {$\Part_k(12)$};
  \draw (30,-13) node[right] {$\Part_l(34)$};
  \draw (25.7,-6) -- (26.3,-6);
  \draw (26,-6) -- (26,-10);
  \draw (25.7,-10) -- (26.3,-10);
  \draw (25.8,-8) node[left] {$k$};
  \draw (27,-4.2) -- (27,-4.8);
  \draw (27,-4.5) -- (32,-4.5);
  \draw (32,-4.2) -- (32,-4.8);
  \draw (29.5,-4.5) node[above] {$l$};
  \draw[thick, blue,->] (16,-3) to[out=35,in=110] node[ above, pos=0.5] {\small $l=\gamma_3-u=\gamma_4-v$} (29.3,-2.7);
  \draw[thick, blue,->] (15,-15) to[out=55,in=-190] node[sloped, above, pos=0.5] {\small $k=\gamma_1-r$} node[sloped, below, pos=0.5] {\small $\ =\gamma_2-s$}(24.4,-8.1);
  \draw[red] (28,-18) node[right] {coupling factor $q^{kl}$};
  \draw[thick, red,->] (35,-17) to[out=80,in=-10] (30,-8); 
\end{tikzpicture}
\end{center}
\caption{The horizontal iterated Durfee decomposition. 
The three rectangles contribute $q^{rs+uv+kl}$. 
The vertical decomposition is the analogous 
picture with the pairings $(1,3)$, $(2,4)$, $(13,24)$.
}
\label{fig:iterated-durfee}
\end{figure}

\begin{thm}[Main combinatorial theorem]
\label{thm:main-bijection}
For every $\gamma\in\N^4$, there is an explicit weight-preserving bijection
\[
  \Phi_\gamma:\mathcal H_\gamma\longrightarrow\mathcal V_\gamma.
\]
It is obtained by following alternating paths of the involutions $I_H$ and
$I_V$ on the common signed set of Section~\ref{s:common-set}.
\end{thm}

Theorems~\ref{thm:Keller-coefficient}, \ref{thm:KellerE} follow immediately from
Theorem~\ref{thm:main-bijection}.  We first record an algebraic proof, then
construct $\Phi_\gamma$.

\section{Proof by generating functions}
\label{s:Keller-algebra}

Put $x_A=\prod_{i\in A}x_i$ for
$A\subseteq\{1,2,3,4\}$.  Summed over all $\gamma$, the horizontal side
of \eqref{eq:Keller} is
\[
  F(x_1,x_2)F(x_3,x_4)F(x_{12},x_{34}).
\]
Lemma~\ref{lem:F-product} turns this into
\begin{equation}
\label{eq:common-product-H}
\frac{(x_{12};q)_\infty}
       {(x_1;q)_\infty(x_2;q)_\infty}
 \frac{(x_{34};q)_\infty}
       {(x_3;q)_\infty(x_4;q)_\infty}
 \frac{(x_{1234};q)_\infty}
       {(x_{12};q)_\infty(x_{34};q)_\infty}
=
\frac{(x_{1234};q)_\infty}
     {\prod_{i=1}^4(x_i;q)_\infty}.
\end{equation}
The vertical side has generating series
\[
  F(x_1,x_3)F(x_2,x_4)F(x_{13},x_{24}),
\]
and the same cancellation gives
\begin{equation}
\label{eq:common-product-V}
  \frac{(x_{1234};q)_\infty}
       {\prod_{i=1}^4(x_i;q)_\infty}.
\end{equation}
Taking the coefficient of
$x_1^{\gamma_1}x_2^{\gamma_2}x_3^{\gamma_3}x_4^{\gamma_4}$
proves \eqref{eq:Keller}.

\begin{remark}
\label{rem:algebra-blueprint}
The proof displays the blueprint for the combinatorics.  Each quotient
\[
  \frac{(AB;q)_\infty}{(A;q)_\infty(B;q)_\infty}
\]
must become a local fixed-point set carrying the rectangle weight.  The
cancelled factors $(x_{12};q)_\infty$, $(x_{34};q)_\infty$, and their
vertical analogues must all be present before either cancellation is
chosen.  This is precisely why the common signed set contains all four
edge colors.
\end{remark}

\section{Algorithm Z and finite Durfee cancellation}
\label{s:algorithm-Z}

We now construct the local involution used throughout the paper.  The
only insertion tool required is Algorithm Z; see
\cite{AndrewsBressoud1984,BressoudZeilberger1989,ChenChenFuZang2010}.

\subsection{Algorithm Z}

Let $\Part_{r,\leq p}$ denote the set of partitions with $r$ parts, every part
at most $p$.

\begin{thm}[Algorithm Z]
\label{thm:algorithm-Z}
For $p,r\geq0$ there is a weight-preserving bijection
\[
  Z_{p,r}:
  \Part_p\times\Part_r
  \longrightarrow
  \Part_{p+r}\times\Part_{r,\leq p}.
\]
\end{thm}

We briefly recall the algorithm.  Insert the parts of
$\beta\in\Part_r$, in decreasing order, into
$\alpha\in\Part_p$.  To insert a part $b$, compare it with the smallest
part of the current insertion partition.  If $b$ is no larger, append it
and record $0$.  Otherwise remove the smallest part, insert $b-1$
recursively, restore the removed part, and increase the record by $1$.
After all insertions, the insertion partition $\mu$ has $p+r$ parts and
the records form a partition $\bar\nu\in\Part_{r,\leq p}$.  Each decrement in an
inserted part is matched by one unit in the record, so
\[
  |\alpha|+|\beta|=|\mu|+|\bar\nu|.
\]
Extraction reverses these steps.

We need the distinct-part version.  Let $
  \Dist_t^{<a}
  =
  \{\nu\in\Dist_t:\nu_1<a\}$.

\begin{cor}
\label{cor:distinct-Z}
For $0\leq t\leq a$, Algorithm Z induces a weight-preserving bijection
\begin{equation}
\label{eq:distinct-Z}
  \Theta_{a,t}:
  \Dist_t\times\Part_{a-t}
  \longrightarrow
  \Part_a\times\Dist_t^{<a}.
\end{equation}
\end{cor}

\begin{proof}
Given $\delta\in\Dist_t$, remove its staircase 
($\rho_t=(t-1,t-2,\ldots,1,0)$), that is,
\[
  \bar\delta=\delta-\rho_t\in\Part_t.
\]
Applying $Z_{a-t,t}$ to $(\alpha,\bar\delta)$ we get
$(\mu,\bar\nu)\in
  \Part_a\times\Part_{t,\leq a-t}$.
Set $\nu=\bar\nu+\rho_t$.  Then $\nu$ is distinct and
\[
  \nu_1\leq(a-t)+(t-1)=a-1.
\]
The staircase removed from $\delta$ has been added to $\bar\nu$, so the
total weight is unchanged.  Every step is reversible.
\end{proof}

\subsection{The local sign-reversing involution}

For two colors $A,B$, consider
\[
  \mathfrak C(A,B)
  =
  \Dist(AB)\times\Part(A)\times\Part(B).
\]
The sign is inherited from the first factor.

\begin{prop}[Finite Durfee cancellation]
\label{prop:local-involution}
There is a weight-preserving sign-reversing involution
\[
  J_{A,B}:\mathfrak C(A,B)\longrightarrow\mathfrak C(A,B)
\]
whose fixed points are naturally
\begin{equation}
\label{eq:local-fixed}
  \bigsqcup_{a,b\geq0}
  \Rect_{a,b}\times\Part_a(A)\times\Part_b(B).
\end{equation}
\end{prop}

\begin{proof}
Take $(\delta,\alpha,\beta)\in\mathfrak C(A,B)$ and put
\[
  t=\ell(\delta),\qquad
  a=\ell(\alpha)+t,\qquad
  b=\ell(\beta)+t.
\]
The monomial part of its weight is $A^aB^b$.  Apply
$\Theta_{a,t}$ to $(\delta,\alpha)$:
\[
  (\delta,\alpha)
  \longleftrightarrow
  (\mu,\nu)
  \in\Part_a\times\Dist_t^{<a}.
\]
We hold $\mu$ fixed and operate on $(\nu,\beta)$.

If a nonnegative integer smaller than $a$ occurs in $\nu$ or $\beta$, let
$j$ be the smallest such integer.  If $j$ occurs in $\nu$, move it from
$\nu$ to $\beta$.  If it does not occur in $\nu$, move one copy from
$\beta$ to $\nu$.  Denote the resulting pair by $(\nu',\beta')$.  The
move preserves weight and changes $\ell(\nu)$ by one.  It is its own
inverse: after the move, the same $j$ is still the smallest eligible part,
and its membership in $\nu$ has been toggled.  Finally apply the inverse
of \eqref{eq:distinct-Z} to $(\mu,\nu')$.  This defines $J_{A,B}$ on all
nonfixed objects and reverses their sign.

There is no eligible $j$ precisely when $\nu=\blank$ and every part of
$\beta$ is at least $a$.  In this case $t=0$, $\delta=\blank$, and
$\alpha=\mu\in\Part_a$.  Subtracting $a$ from each of the $b$ parts of
$\beta$ gives an arbitrary partition in $\Part_b$ and removes a rectangle
of weight $q^{ab}$.  This identifies the fixed points with
\eqref{eq:local-fixed}.
\end{proof}

\begin{remark}
\label{rem:local-product}
At the level of signed generating functions,
Proposition~\ref{prop:local-involution} is the combinatorial identity
\[
  \frac{(AB;q)_\infty}{(A;q)_\infty(B;q)_\infty}
  =
  \sum_{a,b\geq0}q^{ab}\cP_a\cP_bA^aB^b.
\]
This is Lemma~\ref{lem:F-product}, now with an explicit cancellation.
\end{remark}

\subsection{The finite Durfee--BZ bijection}

There is also an elementary cancellation involution on
\[
  \Dist(A)\times\Part(A).
\]
If the two partitions are not both empty, let $j$ be their smallest
occurring part.  Toggle one copy of $j$ between the distinct and ordinary
partition.  Call this involution $T_A$.  Its sole fixed point is
$(\blank,\blank)$.

\begin{proof}[Construction of the bijection \eqref{eq:finite-DBZ-map}]
Consider
\[
  \Dist(AB) \times \Part(A) \times \Part(B) \times \Part(AB).
\]
It has two sign-reversing involutions.  The first applies $T_{AB}$ to the
first and last factors; its fixed points are
$\Part(A)\Part(B)$.  The second applies $J_{A,B}$ to the first three
factors; its fixed points are
\[
  \bigsqcup_{r,s,k\geq0}
  \Rect_{r,s}\Part_r(A)\Part_s(B)\Part_k(AB).
\]
In bidegree $(a,b)$ one has $r+k=a$ and $s+k=b$.  The Garsia--Milne
alternating-path construction \cite{GarsiaMilne1981} gives the
weight-preserving bijection
\eqref{eq:finite-DBZ-map}.  Concretely, start at a fixed point of the
first involution and apply the second and first involutions alternately
until a fixed point of the second is reached.

For completeness, here is the termination argument.  In a fixed
homogeneous component, draw an edge for each nontrivial orbit of either
involution.  Every component of this finite graph is an alternating path
or an alternating cycle.  A fixed point of the first involution has no
edge of the first kind, so it lies at the end of a path, not on a cycle.
Moreover, the fixed points of both involutions have positive sign.  If
the other end of the path were also fixed by the first involution, then
the path would begin and end with edges of the second kind and hence
would have odd length.  Its endpoints would consequently have opposite
signs, a contradiction.  Thus the other endpoint is fixed by the second
involution.  The alternating procedure reaches this endpoint, and
traversing the same path in reverse gives the inverse map.  This is
$\DBZ_{a,b}$.
\end{proof}

\section{The common signed set}
\label{s:common-set}

For each edge color $A\in\{12,34,13,24\}$, the pair
\[
  \Dist(A)\Part(A)
\]
has signed generating function $1$ by \eqref{eq:colored-gen}.  Define
\begin{equation}
\label{eq:common-set}
\begin{split}
\mathfrak S={}&
\Part(1)\Part(2)\Part(3)\Part(4)\Dist(1234)\\
&{}\times
\Dist(12)\Part(12)\Dist(34)\Part(34)\\
&{}\times
\Dist(13)\Part(13)\Dist(24)\Part(24).
\end{split}
\end{equation}
The sign of an object is $(-1)^d$, where $d$ is the total number of parts
in its five distinct-part factors.  Its signed generating function is
\[
  \frac{(x_{1234};q)_\infty}
       {\prod_{i=1}^4(x_i;q)_\infty}.
\]
This is the common product found algebraically in
\eqref{eq:common-product-H} and \eqref{eq:common-product-V}.

For $\gamma\in\N^4$ and $N\in\N$, let
$\mathfrak S_{\gamma,N}$ consist of the objects of monomial weight
$x^\gamma=x_1^{\gamma_1}\cdots x_4^{\gamma_4}$ and $q$-weight $q^N$.

\begin{lem}
\label{lem:homogeneous-finite}
Every set $\mathfrak S_{\gamma,N}$ is finite.
\end{lem}

\begin{proof}
The four color degrees bound the length of every partition occurring in
\eqref{eq:common-set}.  Once the lengths are fixed, there are only
finitely many tuples of nonnegative parts with total sum $N$.
\end{proof}

We will build two involutions on $\mathfrak S$.  We use the following
standard product convention.  Suppose a product of signed sets has
sign-reversing involutions $f_1,\ldots,f_m$ on disjoint blocks.  Scan the
blocks in the displayed order and apply $f_i$ to the first block not fixed
by $f_i$.  If every block is fixed, fix the whole object.  The resulting
\emph{first-moving product involution} is sign reversing: all earlier
blocks stay fixed, and a paired point of the selected block remains
nonfixed, so the same block is selected on the return step.

\section{The horizontal and vertical involutions}
\label{s:HV-involutions}

\subsection{The horizontal involution}

Regroup the factors of \eqref{eq:common-set} into the following five
disjoint blocks, and consider involutions displayed under them:
\begin{equation}\label{eq:H-blocks}
    \begin{array}{ccccc}
    \mathfrak C(1,2) &
    \mathfrak C(3,4) &
    \mathfrak C(12,34) &
    \Dist(13)\Part(13) &
    \Dist(24)\Part(24) 
    \\
    J_{1,2} &
    J_{3,4} &
    J_{12,34} &
    T_{13} &
    T_{24}.
    \end{array}
\end{equation}
Define $I_H$ to be their first-moving product involution. 

\begin{prop}
\label{prop:H-fixed}
For every $\gamma$, the fixed points of $I_H$ of color degree $\gamma$
are naturally identified, weight for weight, with $\mathcal H_\gamma$.
\end{prop}

\begin{proof}
All five blocks in \eqref{eq:H-blocks} must be fixed.  The last two are
therefore empty.  By Proposition~\ref{prop:local-involution}, the first
block contributes
$\Rect_{r,s}\Part_r(1)\Part_s(2)$,
the second contributes
$\Rect_{u,v}\Part_u(3)\Part_v(4)$,
and the third contributes $\Rect_{k,l}\Part_k(12)\Part_l(34)$.
The four color degrees are
\[
  r+k,\qquad s+k,\qquad u+l,\qquad v+l.
\]
Thus color degree $\gamma$ is equivalent to
$(r,s,k;u,v,l)\kp_H\gamma$.  The three rectangles contribute
$q^{rs+uv+kl}$, and the six partitions contribute the six factors
$\cP$.  This is exactly $\mathcal H_\gamma$.
\end{proof}

\subsection{The vertical involution}
Now regroup the factors of \eqref{eq:common-set} into the following five disjoint blocks, and consider involutions displayed under them:
\begin{equation}\label{eq:V-blocks}
    \begin{array}{ccccc}
    \mathfrak C(1,3) &
    \mathfrak C(2,4) &
    \mathfrak C(13,24) &
    \Dist(12)\Part(12) &
    \Dist(34)\Part(34) 
    \\
    J_{1,3} &
    J_{2,4} &
    J_{13,24} &
    T_{12} &
    T_{34}.
    \end{array}
\end{equation}
Define $I_V$ to be their first-moving product involution. 

\begin{prop}
\label{prop:V-fixed}
For every $\gamma$, the fixed points of $I_V$ of color degree $\gamma$
are naturally identified, weight for weight, with $\mathcal V_\gamma$.
\end{prop}
The proof is identical to the horizontal argument after replacing the
pairings $(1,2),(3,4),(12,34)$ by
$(1,3),(2,4),(13,24)$ and the toggles $T_{13},T_{24}$ by
$T_{12},T_{34}$.  The three applications of $J$ produce precisely the
vertical block data in \eqref{eq:V-blocks}, while the two toggles force
the remaining paired factors to be empty.

\begin{remark}
Every fixed point of $I_H$ or $I_V$ has positive sign.  Indeed, all
distinct-part factors in a fixed point are empty.  This simple observation
also makes the direction of the alternating paths transparent.
\end{remark}

\section{The alternating-path bijection}
\label{s:alternating-path}

We finish the construction of Theorem~\ref{thm:main-bijection}.  Fix
$h\in\mathcal H_\gamma$ and regard it, through
Proposition~\ref{prop:H-fixed}, as an $I_H$-fixed point of
$\mathfrak S$.  This is the involution principle of
\cite{GarsiaMilne1981}.  If $h$ is fixed by $I_V$, set
$\Phi_\gamma(h)=h$.
Otherwise form the path
\begin{equation}
\label{eq:alternating-path}
  h
  \xmapsto{I_V}
  I_V(h)
  \xmapsto{I_H}
  I_HI_V(h)
  \xmapsto{I_V}
  I_VI_HI_V(h)
  \xmapsto{I_H}\cdots.
\end{equation}
Stop at the first $I_V$-fixed point, and identify it with an element of
$\mathcal V_\gamma$ by Proposition~\ref{prop:V-fixed}.

\begin{prop}
\label{prop:path-terminates}
The path \eqref{eq:alternating-path} terminates.  Its endpoint is
$I_V$-fixed, and reversing the path gives the inverse construction.
\end{prop}

\begin{proof}
Both involutions preserve color and $q$-weight.  Hence the path remains in
one finite set $\mathfrak S_{\gamma,N}$ by
Lemma~\ref{lem:homogeneous-finite}.  Draw an $H$-edge between the two
members of every nontrivial $I_H$-orbit and a $V$-edge for $I_V$.
Each vertex has at most one edge of each kind, so every component is an
alternating path or an alternating cycle.

The starting vertex $h$ has no $H$-edge.  It therefore lies in a path
component, not a cycle.  If its other endpoint were also $I_H$-fixed, the
path would contain an odd number of edges, so that endpoint would have
negative sign.  This is impossible because every $I_H$-fixed point has
positive sign.  Thus the other endpoint has no $V$-edge and is
$I_V$-fixed.  Traversing the same path from that endpoint, alternating
$I_H$ and $I_V$, returns to $h$.
\end{proof}

\begin{proof}[Proof of Theorem~\ref{thm:main-bijection}]
Define $\Phi_\gamma$ by the endpoint rule above.  Every step of the path
preserves weight, and Proposition~\ref{prop:path-terminates} gives an
inverse.  Hence $\Phi_\gamma$ is a weight-preserving bijection
$\mathcal H_\gamma\to\mathcal V_\gamma$.
\end{proof}

\begin{example}[An alternating path]
\label{ex:alternating-path}
Let $\gamma=(1,1,1,1)$ and start with the element
$h\in\mathcal H_\gamma$ having parameters
\[
  (r,s,k;u,v,l)=(0,0,1;1,1,0)
\]
and six partition entries, in the order $(r,s,k,u,v,l)$,
\[
  (\blank,\blank,(0),(0),(0),\blank).
\]
Thus its only nontrivial rectangle is the $1$-by-$1$ rectangle
$\Rect_{u,v}$.  Under the fixed-point identification of
Proposition~\ref{prop:H-fixed}, the corresponding object of
$\mathfrak S$ has the following nonempty ordinary-partition factors:
\[
  \lambda^3=(0),\qquad
  \lambda^4=(1),\qquad
  \lambda^{12}=(0).
\]
Here the part $1$ in $\lambda^4$ incorporates the rectangle shift
between the colors $3$ and $4$.

For brevity, in the path below we list only the nonempty partition
factors; $\lambda$ denotes an ordinary partition and $\delta$ a
distinct-part partition.  The two local $J$-moves used in the path are
\[
\begin{aligned}
  J_{1,2}\bigl((0)^{12},\blank,\blank\bigr)
  &=
  \bigl(\blank,(0)^1,(0)^2\bigr),\\
  J_{1,3}\bigl(\blank,(0)^1,(0)^3\bigr)
  &=
  \bigl((0)^{13},\blank,\blank\bigr).
\end{aligned}
\]
Write the successive objects as
\[
\begin{aligned}
 S_0&=\langle\lambda^3=(0),\lambda^4=(1),
             \lambda^{12}=(0)\rangle,\\
 S_1&=\langle\lambda^3=(0),\lambda^4=(1),
             \delta^{12}=(0)\rangle,\\
 S_2&=\langle\lambda^1=(0),\lambda^2=(0),
             \lambda^3=(0),\lambda^4=(1)\rangle,\\
 S_3&=\langle\delta^{13}=(0),\lambda^2=(0),
             \lambda^4=(1)\rangle,\\
 S_4&=\langle\lambda^{13}=(0),\lambda^2=(0),
             \lambda^4=(1)\rangle.
\end{aligned}
\]
Then the alternating path is
\[
\begin{aligned}
 S_0&\xmapsto{\,I_V:T_{12}\,}S_1
      \xmapsto{\,I_H:J_{1,2}\,}S_2,\\
 S_2&\xmapsto{\,I_V:J_{1,3}\,}S_3
      \xmapsto{\,I_H:T_{13}\,}S_4.
\end{aligned}
\]
The last object $S_4$ is fixed by $I_V$: its $(1,3)$ block is empty, its
$(2,4)$ block carries the $1$-by-$1$ rectangle, and its $(13,24)$
block has lengths $(1,0)$.  Hence it represents the element
$\Phi_\gamma(h)\in\mathcal V_\gamma$ with vertical parameters
\[
  (r,s,k;u,v,l)=(0,0,1;1,1,0)
\]
and the same six numerical partition entries as $h$, now attached to
the vertical colors.  Every object on the path has weight
$x_1x_2x_3x_4q$.
\end{example}

\begin{proof}[Combinatorial proof of
Theorem~\ref{thm:Keller-coefficient}]
The two sides of \eqref{eq:Keller} are the weight generating functions of
$\mathcal H_\gamma$ and $\mathcal V_\gamma$.  They are equal by
Theorem~\ref{thm:main-bijection}.
\end{proof}

\begin{remark}
The algebraic cancellation in Section~\ref{s:Keller-algebra} and the
combinatorial cancellation here are term-for-term analogues.  The
horizontal factorization chooses the three binary nodes
\[
  (1,2),\quad(3,4),\quad(12,34),
\]
while the vertical factorization chooses
\[
  (1,3),\quad(2,4),\quad(13,24).
\]
The common signed set retains both choices simultaneously.  The
Garsia--Milne path is the mechanism that changes one binary factorization
into the other without passing through the signed generating function.
\end{remark}

\section{A standard-monomial interpretation}
\label{s:standard-monomials}

This section is logically independent of the preceding combinatorial
proof.  We prove one binary normal-form theorem and apply it to the
pentagon identity, the uncoupled identity, and Keller's identity.
The arc-space framework and the binary node degeneration used here are
known; the new point is their four-color multigraded and iterated use.
More precisely, we construct two standard-monomial bases, associated
with the pairings $(12),(34)$ and $(13),(24)$, and identify their
multidegree enumerators with the two sides of Keller's identity.  The
third binary factor in each construction also gives a direct
standard-monomial explanation of the coupling exponent $kl$.

\subsection{The binary Durfee normal form}

The quotient considered below is the focused arc algebra of the node
$xy=0$, up to the index shift between arcs centered at the origin and
series beginning in degree zero.  Its Gr\"obner degeneration and
two-colored standard monomials are a reindexed form of the node theorem
attributed to Nguyen Duc Tam \cite{Nguyen2015} and used by
Afsharijoo--Mourtada \cite{AfsharijooMourtada2020}.  We include a direct proof adapted to the
iterated normal-form construction required below.

Let
\[
  R_{A,B}=\C[a_0,a_1,\ldots,b_0,b_1,\ldots],
\]
and introduce
\[
  A(t)=\sum_{i\geq0}a_it^i,\qquad
  B(t)=\sum_{j\geq0}b_jt^j.
\]
Write
\[
  A(t)B(t)=\sum_{s\geq0}f_st^s,\qquad
  f_s=\sum_{i+j=s}a_i b_j,
\]
and set $S_{A,B}=\C[f_0,f_1,\ldots]$.  Give the variables the
multidegrees
\[
  \deg a_i=Aq^i,\qquad \deg b_j=Bq^j;
\]
then $\deg f_s=ABq^s$.

Let $\cB(A,B)$ be the span of the monomials
\begin{equation}
\label{eq:binary-standard-monomials}
  a_{i_1+n}\cdots a_{i_m+n}
  b_{j_1}\cdots b_{j_n},
  \qquad
  \begin{matrix}
  0\leq i_1\leq\cdots\leq i_m,\\
  0\leq j_1\leq\cdots\leq j_n.
  \end{matrix}
\end{equation}
The shift of the $a$-indices contributes $q^{mn}$; hence
\[
  \Hilb\cB(A,B)
  =
  \sum_{m,n\geq0}q^{mn}\cP_m\cP_nA^mB^n
  =F(A,B).
\]

\begin{thm}[Binary Durfee normal form]
\label{thm:binary-normal-form}
Multiplication is an isomorphism of multigraded vector spaces
\begin{equation}
\label{eq:binary-normal-form}
  \cB(A,B)\otimes_{\C}S_{A,B}
  \xrightarrow{\ \sim\ }R_{A,B}.
\end{equation}
Thus the monomials \eqref{eq:binary-standard-monomials} give a basis of
$R_{A,B}/(f_0,f_1,\ldots)$.
\end{thm}

\begin{proof}
We give the adapted proof for completeness.  We first compute the
quotient.  For fixed $N$, work in
\[
  R_N=\C[a_0,\ldots,a_N,b_0,\ldots,b_N].
\]
The equations $f_0=\cdots=f_N=0$ say
\[
  A_N(t)B_N(t)\equiv0\pmod {t^{N+1}},
\]
where $A_N$ and $B_N$ are the truncations of $A$ and $B$.  If
$\alpha=\operatorname{ord}A_N$ and
$\beta=\operatorname{ord}B_N$, with order $N+1$ assigned to the zero
polynomial, then $\alpha+\beta\geq N+1$.  The corresponding stratum has
dimension at most
\[
  (N+1-\alpha)+(N+1-\beta)\leq N+1.
\]
The locus $A_N=0$ has dimension $N+1$, so
$(f_0,\ldots,f_N)$ has height $N+1$.  This ideal is generated by
$N+1$ elements in the Cohen--Macaulay ring $R_N$ and has height
$N+1$; it is therefore a complete intersection, and
$f_0,\ldots,f_N$ is a regular sequence.  Adjoining the remaining
variables preserves regularity.  Since this holds for every $N$, the
sequence $f_0,f_1,\ldots$ is regular in $R_{A,B}$.  Also the $f_s$ are algebraically
independent: under the specialization $b_0=1$, $b_j=0$ for $j>0$, one
has $f_s=a_s$.  Therefore, with $I=(f_0,f_1,\ldots)$,
\begin{equation}
\label{eq:Hilb-binary-quotient}
  \Hilb(R_{A,B}/I)
  =
  \frac{(AB;q)_\infty}
       {(A;q)_\infty(B;q)_\infty}
  =F(A,B).
\end{equation}

We next identify a basis of the quotient.  Fix a multihomogeneous
component and choose $N$ large enough that all variables occurring in
that component belong to
$\C[a_0,\ldots,a_N,b_0,\ldots,b_N]$.  On this finite polynomial ring
use the lexicographic monomial order
\[
  a_0>a_1>\cdots>a_N>b_0>b_1>\cdots>b_N.
\]
All comparisons below take place in finite multihomogeneous
components, so these compatible degreewise orders suffice.  Let $J$ be
the monomial ideal generated by
\begin{equation}
\label{eq:binary-forbidden}
  a_i b_{j_0}b_{j_1}\cdots b_{j_i}
\end{equation}
$(i,j_0,\ldots,j_i\geq0)$.  An induction on $i$ shows that every such
monomial is a leading term of an element of $I$, and hence
$J\subseteq\operatorname{in}(I)$.  For $i=0$ this follows from
$f_{j_0}$.  For the induction step, expand
\[
  f_{i+j_0}b_{j_1}\cdots b_{j_i}.
\]
Its $a_i$-term is \eqref{eq:binary-forbidden}.  Every term with
$a$-index smaller than $i$ is divisible by a forbidden monomial already
obtained, and after reducing these terms, the $a_i$-term precedes all
remaining terms.

The monomials outside $J$ are precisely those containing $n$ factors
$b_j$ and only factors $a_i$ with $i\geq n$, namely the monomials
\eqref{eq:binary-standard-monomials}.  Therefore
\[
  \Hilb(R_{A,B}/J)=F(A,B)
  =\Hilb(R_{A,B}/I)
  =\Hilb(R_{A,B}/\operatorname{in}(I)),
\]
where the middle equality is \eqref{eq:Hilb-binary-quotient} and an
initial ideal preserves the multigraded Hilbert series.  Since
$J\subseteq\operatorname{in}(I)$ and every multihomogeneous component is
finite dimensional, equality of the Hilbert series forces
$J=\operatorname{in}(I)$.  Thus the residue classes of
\eqref{eq:binary-standard-monomials} form a basis of $R_{A,B}/I$.

It remains to lift from the quotient.  The basis just obtained shows,
by induction on the total $A,B$-degree, that multiplication in
\eqref{eq:binary-normal-form} is surjective: after subtracting a standard
representative, the remainder is a finite sum $\sum_s f_sh_s$, and each
$h_s$ has smaller total degree.  Finally,
\[
  \Hilb\bigl(\cB(A,B)\otimes S_{A,B}\bigr)
  =
  F(A,B)\frac{1}{(AB;q)_\infty}
  =
  \frac{1}{(A;q)_\infty(B;q)_\infty}
  =
  \Hilb R_{A,B}.
\]
The multihomogeneous components are finite dimensional, so the
surjection is an isomorphism.
\end{proof}

\begin{remark}
The normal form is asymmetric: the $a$-indices are shifted by the number
of $b$-factors.  Reversing the leading-term order shifts the $b$-indices
instead.  These are the two orientations of the same Durfee rectangle.
\end{remark}

\subsection{The pentagon and uncoupled identities}

The two factors in \eqref{eq:binary-normal-form} have Hilbert series
\[
  \Hilb\cB(A,B)
  =
  \sum_{r,s\geq0}q^{rs}\cP_r\cP_sA^rB^s,
  \qquad
  \Hilb S_{A,B}
  =
  \sum_{k\geq0}\cP_k(AB)^k.
\]
The coefficient of $A^aB^b$ in
$\Hilb R_{A,B}=\Hilb\cB(A,B)\Hilb S_{A,B}$ is exactly
\eqref{eq:pentagon}.  Thus Theorem~\ref{thm:binary-normal-form} is a
standard-monomial proof of the pentagon identity.

For four colors, let
\[
  Z_i(t)=\sum_{r\geq0}z_{i,r}t^r,\qquad
  R=\C[z_{i,r}:1\leq i\leq4,\ r\geq0],
  \qquad
  \deg z_{i,r}=x_iq^r.
\]
Write
\[
  Z_i(t)Z_j(t)=\sum_{s\geq0}f_s^{ij}t^s,
  \qquad
  S_{ij}=\C[f_0^{ij},f_1^{ij},\ldots].
\]
Two independent applications of the binary theorem give the horizontal
normal form
\[
  R\cong
  \cB(1,2)\otimes S_{12}\otimes
  \cB(3,4)\otimes S_{34},
\]
and the vertical normal form
\[
  R\cong
  \cB(1,3)\otimes S_{13}\otimes
  \cB(2,4)\otimes S_{24}.
\]
Taking the $\gamma$-graded Hilbert series of these two bases gives,
respectively, the two sides of \eqref{eq:uncoupled}.  Thus the uncoupled
identity compares two standard-monomial bases of the same four-colored
polynomial ring.

\subsection{Keller's identity and the four-colored arc algebra}

Expand
\[
  Z_1(t)Z_2(t)Z_3(t)Z_4(t)
  =
  \sum_{s\geq0}g_st^s
\]
and define the four-colored arc algebra
\begin{equation}
\label{eq:arc-algebra}
  R^{1234}=R/(g_0,g_1,\ldots).
\end{equation}
If
\[
  x_i(t)=\sum_{r\geq1}x_i^{(r)}t^r
\]
denotes an arc centered at the origin and we set
$z_{i,r}=x_i^{(r+1)}$, then $x_i(t)=tZ_i(t)$.  Consequently
\eqref{eq:arc-algebra} is precisely the focused arc algebra at the
origin of the normal-crossing hypersurface
\[
  x_1x_2x_3x_4=0.
\]
Bruschek--Mourtada--Schepers developed this focused arc algebra
framework for normal crossings, including the regular-sequence method
for the coefficient equations and the resulting ordinary
Hilbert--Poincar\'e series \cite{BruschekMourtadaSchepers2013}; see also
Mourtada's survey \cite{Mourtada2023}.  In the present notation their
ordinary series is recovered from the four-color series below by
specializing $x_1=x_2=x_3=x_4=q$:
\[
  \frac{(q^4;q)_\infty}{(q;q)_\infty^4}.
\]
Our purpose is to retain the four separate color degrees and construct
the two iterated binary bases that refine this known series.

Starting with the pairs $(1,2)$ and $(3,4)$ gives
\[
  R\cong
  \cB(1,2)\otimes\cB(3,4)\otimes S_{12}\otimes S_{34}.
\]
Inside the last two factors, the series
$\sum f_s^{12}t^s$ and $\sum f_s^{34}t^s$ have product
$\sum g_st^s$.  A third application of
Theorem~\ref{thm:binary-normal-form} therefore gives
\[
  S_{12}\otimes S_{34}
  \cong
  \cB(12,34)\otimes\C[g_0,g_1,\ldots].
\]
Combining the three multiplication isomorphisms gives a
$\C[g_0,g_1,\ldots]$-linear isomorphism
\[
  \cB(1,2)\otimes\cB(3,4)\otimes\cB(12,34)
  \otimes\C[g_0,g_1,\ldots]
  \xrightarrow{\ \sim\ }R.
\]
Thus $R$ is free as a module over $\C[g_0,g_1,\ldots]$, with the first
three factors as a multigraded vector-space basis.  Tensoring this
decomposition over $\C[g_0,g_1,\ldots]$ with
$\C[g_0,g_1,\ldots]/(g_0,g_1,\ldots)\cong\C$, or equivalently
quotienting by the ideal generated by the $g_s$, gives
\begin{equation}
\label{eq:Keller-H-normal-form}
  R^{1234}\cong
  \cB(1,2)\otimes\cB(3,4)\otimes\cB(12,34).
\end{equation}
Its Hilbert series is
\[
  F(x_1,x_2)F(x_3,x_4)F(x_{12},x_{34}),
\]
and its $\gamma$-graded component is the left-hand side of
\eqref{eq:Keller}.  In particular, the third normal form supplies the
coupling factor $q^{kl}$: it is the Durfee shift for the node formed by
the two intermediate series $Z_1Z_2$ and $Z_3Z_4$.

Starting instead with $(1,3)$ and $(2,4)$ gives
\begin{equation}
\label{eq:Keller-V-normal-form}
  R^{1234}\cong
  \cB(1,3)\otimes\cB(2,4)\otimes\cB(13,24),
\end{equation}
whose $\gamma$-graded Hilbert series is the right-hand side of
\eqref{eq:Keller}.  Keller's identity therefore equates two iterated
Durfee normal-form bases of the same algebra; here the coupling factor
$q^{kl}$ is supplied by the node formed by the intermediate series
$Z_1Z_3$ and $Z_2Z_4$.  The common Hilbert series is
\[
  \Hilb R^{1234}
  =
  \frac{(x_{1234};q)_\infty}
       {\prod_{i=1}^4(x_i;q)_\infty},
\]
the same expression that underlies both the product proof and the common
signed set.

Thus this section does not claim a new construction of the focused arc
algebra, a new regular-sequence computation of its ordinary Hilbert
series, or a new binary node degeneration.  Its contribution is the
four-color refinement, the two iterated decompositions
\eqref{eq:Keller-H-normal-form} and \eqref{eq:Keller-V-normal-form}, and
the identification of their coefficient equality---including the third
node contribution $q^{kl}$---with the $A_2\square A_2$ Keller identity.

\bigskip
\hrule
\bigskip

\noindent During the preparation of this work the author used Generative AI in order to improve language, to check and refine proofs, and debug computer code. After using this tool, the author reviewed and edited the content as needed and takes full responsibility for the content of the publication.

\bigskip
\hrule
\bigskip

\bibliographystyle{amsplain}
\bibliography{KellerA2A2}

\end{document}